\documentclass{amsart}

\usepackage{amssymb,array}
\usepackage{amsmath}
\usepackage{amsthm}
\usepackage{amsbsy,mathrsfs}
\usepackage{bm}
\usepackage[english]{babel}
\usepackage{graphicx}
\usepackage{color}
\usepackage{graphicx}
\usepackage{color}
\usepackage{hyperref}
\usepackage[toc,page]{appendix}
\usepackage{thmtools}
\usepackage{thm-restate}

\usepackage{cleveref}

\theoremstyle{plain}
\newtheorem{theorem}{Theorem}[section]
\newtheorem*{theorem*}{Theorem}

\newtheorem*{question*}{Question}
\newtheorem*{main*}{Main Theorem}
\theoremstyle{definition}
\newtheorem{definition}[theorem]{Definition}

\newcommand{\K}{\mathbb{K}}
\newcommand{\Hom}{\operatorname{Hom}}
\newcommand{\tr}{\operatorname{tr}}

\title[Non-injectivity of trace map]{Non-Injectivity of the trace map for character varieties}
\author{Deblina Das}
 \address{Department of Mathematics, Indian Institute of Technology Palakkad}
 \email{212114002@smail.iitpkd.ac.in}
 \author{Arpan Kabiraj}
 \address{Department of Mathematics, Indian Institute of Technology Palakkad}
\email{arpaninto@iitpkd.ac.in}

\begin{document}

\maketitle
\begin{abstract}
    
    Given a closed oriented surface $\Sigma$ of genus at least two, the Goldman trace map defines a function from the vector space generated by the free homotopy classes of oriented closed curves to the Poisson algebra of regular functions on the $G$-character variety where $G$ is a reductive (real or complex) linear Lie group. In this article, we prove that this map is never injective. For each $n$, we construct an explicit nonzero element of the vector space  whose associated trace function vanishes on every homomorphism from  $\pi_1(\Sigma)$ to $GL_n$. The construction is based on the Amitsur-Levitzki identity, together with a choice of words in a free subgroup of $\pi_1(\Sigma)$, ensuring that no cancellation occurs at the level of free homotopy classes. This gives a uniform family of explicit kernel elements, proving Goldman’s predicted non-injectivity of the trace map in arbitrary rank.
\end{abstract}
\section{Introduction}

Let $\Sigma$ be a closed oriented surface of genus $g\geq 2$ with the fundamental group $\pi_1$. 
Let ${\pi}$ be the set of free homotopy classes of oriented
closed curves on $\Sigma$, equivalently the set of conjugacy classes in $\pi_1(\Sigma)$. Given $\alpha\in \pi_1$, we denote its free homotopy class by $|\alpha|$.  Unless otherwise specified, we assume all the objects to be over the ring of real or complex numbers, which is denoted by $\K$. We denote the vector space generated by $\pi$ over $\K$ with $\K\pi.$

Let $G$ be a reductive linear Lie group. Consider  $\Hom(\pi_1,G)/G,$ the set of all homomorphisms from $\pi_1$ to $G$ quotiented by the conjugation action in the sense of invariant theory, known as the $G$-character variety (see \cite{MR2931326}).  In his seminal papers \cite{goldman_invariant_1986}, \cite{MR762512}, Goldman introduced and studied a symplectic structure on the $G$- character variety $\Hom(\pi_1,G)/G$ which generalizes both the Atiyah-Bott and Weil-Petersson symplectic structures, thereby obtaining a Poisson algebra structure on the regular functions on $ G$-character variety. For every $\alpha\in \pi_1$, we have a regular map $\tr_\alpha:\Hom(\pi_1,G)/G\to \K$ defined by $\tr_\alpha(\rho)=\text{trace}(\rho(\alpha))$. Observe that $\tr_\alpha$ is well defined for $|\alpha|$. 
\begin{definition}
   The map $\tr$ from $\K\pi $ to the Poisson algebra of regular functions on the $G$-character variety defined by $\tr(|\alpha|)=\tr_{|\alpha|}$ is called the \emph{Goldman trace map}.
\end{definition}

Goldman explicitly computed the Poisson bracket of trace functions in terms of
the intersection data of the corresponding curves \cite{goldman_invariant_1986}. In particular, for 
$G=GL_n$, the bracket has the following form
$$
\{\tr_\alpha,\tr_\beta\}
=
\sum_{p\in \alpha\cap\beta}
\epsilon(p;\alpha,\beta)\tr_{\alpha*_p\beta},
$$
where $\epsilon(p;\alpha,\beta)$ is the sign at the intersection point $p$ and 
$\alpha*_p\beta$ is the loop product  of $\alpha$ and $\beta$ at $p$. This leads Goldman to the following definition of a Lie bracket on $\K\pi$ known as the \emph{Goldman Lie bracket}: $$[|\alpha|,|\beta|]=\sum_{p\in\alpha\cap\beta}\epsilon(\alpha,\beta)|\alpha*_p\beta|. $$ 
Therefore for $GL_n$, the Goldman trace map is a Lie algebra homomorphism. Goldman’s construction raises a natural and compelling question: 
\begin{question*}Is the Goldman trace map injective for all reductive  linear Lie groups $G$. 
\end{question*}
Goldman anticipated that the trace map is never injective \cite[Section 5]{goldman_invariant_1986}. More precisely, he wrote that 
    ``this homomorphism seems never to be injective:  linear relations between traces of words in $G$ gives rise to elements of the kernel".
In low rank, this phenomenon is already visible
from classical trace identities. For instance, for $A\in SL_2$, the Cayley-Hamilton
theorem gives the following trace relation
$$
{\tr}(AB)+{\tr}(A^{-1}B)
=
{\tr}(A){\tr}(B)
.$$
 Such identities are closely related to the existence of non-conjugate words in free groups which have the same trace under all
$SL_2$-representations. This topic has a substantial mathematical history. Horowitz \cite{MR314993} studied the problem of when two elements of a free group have the same character under all
representations into $SL_2$, and constructed large families of non-conjugate
$SL_2$-trace-equivalent words. See \cite{MR2044556},\cite{MR2272719}, \cite{MR2026843}, \cite{MR2272138}, \cite{MR1622287} and the references therein for more details and recent developments in this direction. 

A tempting approach to solve the higher rank problem would be to consider the  set of non-conjugate elements of $SL_2$-trace-equivalent words mentioned above. However, the higher rank case becomes much more difficult. Firstly, the special
trace identities available for $SL_2$ do not have equally simple analogs for $SL_n$ or $GL_n$ when $n\geq 3$. This distinction is reflected in the work of Lawton, Louder, and McReynolds in \cite{MR3641838},
where the authors discussed the difficulty of finding $SL_n$-trace-equivalent pairs in free
groups for $n>2$. In particular, the classical pairs arising from Horowitz's
construction need not remain trace-equivalent in rank three, and computational
evidence suggests that such pairs are much harder to find in higher rank. 

The purpose of this article is to prove Goldman's predicted non-injectivity of the
trace map for Linear Lie groups of arbitrary rank. Rather than searching
for a pair of non-conjugate words with identical trace, we use a universal
polynomial identity for matrices to produce a nontrivial \emph{linear combination}
of free homotopy classes whose trace function vanishes identically. The identity
we use is the Amitsur-Levitzki theorem \cite{MR36751}:
\begin{theorem*}[Amitsur-Levitzki theorem \cite{MR36751}]
Let $M_n(\K)$ be the ring of $n\times n$
matrices over $\K$. For all
$X_1,\ldots,X_{2n}\in M_n(\K)$, we have
$$
\sum_{\sigma\in S_{2n}}
\operatorname{sgn}(\sigma)
X_{\sigma(1)}X_{\sigma(2)}\cdots X_{\sigma(2n)}
=0
$$ where $\operatorname{sgn}(\sigma)$ is the sign of the permutation $\sigma$. 
\end{theorem*}

Our main result is the following.
\begin{main*}\label{thm:main}
Let $\Sigma$ be a closed oriented surface of genus $g\geq 2$. For every $n\geq 1$, the Goldman trace map is not injective.
\end{main*}

The idea of the proof is to convert the Amitsur-Levitzki identity into a
linear relation among Goldman trace functions. We choose an embedded one-holed torus $Y\subset \Sigma$ and write $\pi_1(Y)=F(a,b), $ the free group generated by $a$ and $b$. For any positive integer $n$, we construct $2n$ words whose alternating products are pairwise non-conjugates in $F(a,b)$. We then use the Amitsur-Levitzki theorem together with a standard fact about conjugacy of two elements in amalgamated products to conclude our theorem. 

Our construction differs from the classical trace-equivalent word approach in an essential way. We do not produce two non-conjugate elements with the same
trace under every representation. Instead, we produce an alternating sum of many distinct conjugacy classes whose associated trace functions cancel because of a
universal matrix identity. 

\section{Proof of the Main Theorem \ref{thm:main}}

\begin{proof}
Let $Y$ be an embedded one-holed torus
in $\Sigma$ and  $\pi_1(Y)=F(a,b)$
be the free group generated by two elements $a$ and $b$. Put $m=2n$
and define
$$
x_i=b^i a\in \pi_1(Y) \quad 
\text{for} \quad 1\leq i\leq m. $$
Now consider 
$$
\Theta_n=\sum_{\sigma\in S_m}\operatorname{sgn}(\sigma)\,\left|a\,x_{\sigma(1)}x_{\sigma(2)}\cdots x_{\sigma(m)}\right|\in \K{\pi}.$$
Let $\rho:\pi\longrightarrow GL_n$
be any representation. Set
$ A=\rho(a),$ and $
X_i=\rho(x_i),$ for $1\leq i\leq m.$
By the Amitsur--Levitzki theorem, 
$$
\sum_{\sigma\in S_{2n}}
\operatorname{sgn}(\sigma)
X_{\sigma(1)}X_{\sigma(2)}\cdots X_{\sigma(2n)}
=0.
$$
Multiplying on the left by $A$ and taking trace, we obtain
$$
\sum_{\sigma\in S_{2n}}
\operatorname{sgn}(\sigma)
\operatorname{tr}\left(
A X_{\sigma(1)}X_{\sigma(2)}\cdots X_{\sigma(2n)}
\right)
=0.
$$
However, observe that
$$
A X_{\sigma(1)}X_{\sigma(2)}\cdots X_{\sigma(2n)}
=\rho\left(a\,x_{\sigma(1)}x_{\sigma(2)}\cdots x_{\sigma(2n)}\right).
$$

Since $\rho$ was arbitrary, it follows that $\Theta_n$ is in the kernel of the trace map.  

It remains to show that $\Theta_n\neq 0$ in $\K{\pi}$.

For $\sigma\in S_m$, write
$$W_\sigma=a\,x_{\sigma(1)}x_{\sigma(2)}\cdots x_{\sigma(m)}.$$
Since $x_i=b^i a$, this word has the form
$$W_\sigma=a\,b^{\sigma(1)}a\,b^{\sigma(2)}a
\cdots a\,b^{\sigma(m)}a.$$
We claim that the conjugacy classes $|W_\sigma|$ are pairwise distinct in $\pi$.

First, we prove that they are distinct within the free group $F(a,b)=\pi_1(Y)$. Each $W_\sigma$ is a
cyclically reduced word in $F(a,b)$. In a free group, two cyclically reduced words
are conjugate if and only if one is a cyclic permutation of the other. The word $W_\sigma$ has a cyclic sequence of $b$-block lengths
$$
\sigma(1),\sigma(2),\ldots,\sigma(m),0.
$$
Here the final $0$ records the passage from the last letter $a$ of $W_\sigma$ back
to the first letter $a$ in the cyclic word. This $0$ occurs exactly once. Hence, each
cyclic sequence corresponding to $W_\sigma$ has a distinguished position, namely the position of $0$. Therefore, if $W_\sigma$ and
$W_\tau$ are conjugate in $F(a,b)$, their cyclic block sequences must agree after
a cyclic rotation, and the unique $0$ must align with the unique $0$. It follows
that
$$
\sigma(1)=\tau(1),\quad
\sigma(2)=\tau(2),\quad\ldots,\quad
\sigma(m)=\tau(m),
$$
 hence $\sigma=\tau$. Thus, elements $W_\sigma$ are pairwise non-conjugate
in $\pi_1(Y)$.

We now show that they remain pairwise non-conjugate in the whole surface group
$\pi_1$. Let
$Z=\overline{\Sigma\setminus Y}.
$ Since $Y$ is an embedded incompressible one-holed torus, the Seifert--van Kampen
theorem gives an amalgamated product decomposition
$$
\pi_1(\Sigma)
\cong
\pi_1(Y)*_{\langle \partial Y\rangle}\pi_1(Z),
$$
where $\langle \partial Y\rangle$ is the cyclic subgroup generated by the boundary
curve of $Y$.

Observe that no $W_\sigma$ is conjugate to some element of 
$\langle\partial Y\rangle$ inside $\pi_1(Y)$. Indeed, with a suitable choice of orientation,
the boundary element of the one-holed torus is represented by the commutator $[a,b].$
Therefore, every element conjugate to a power of $\partial Y$ has trivial image in $
H_1(Y;\mathbb Z)\cong \mathbb Z[a]\oplus \mathbb Z[b].$
On the other hand,
$[W_\sigma]=(m+1)[a]+
\left(\sum_{i=1}^{m} i\right)[b]$
in $H_1(Y;\mathbb Z)$. Since $m=2n\geq 2$, this is nonzero. Hence $W_\sigma$ is not
conjugate into $\langle\partial Y\rangle$.

Now suppose that $W_\sigma$ and $W_\tau$ are conjugate in $\pi_1$. Since
both lie in $\pi_1(Y)$ and neither is conjugate into $\langle\partial Y\rangle$,
by  \cite[Theorem 4.6]{MR422434}, $W_\sigma$ and
$W_\tau$ are already conjugate in $\pi_1(Y)$, 
which implies $\sigma=\tau$.
\end{proof}

\bibliographystyle{amsplain}

\bibliography{gla.bib}

\end{document}